\documentclass[11pt]{article}

\usepackage[T1]{fontenc}
\usepackage{lmodern}
\usepackage{amsmath,amssymb,amsthm,mathtools}
\usepackage{booktabs}
\usepackage{microtype}
\usepackage[a4paper,margin=28mm]{geometry}
\usepackage{xcolor}
\usepackage{hyperref}
\hypersetup{
  colorlinks=true
}

\newtheorem{theorem}{Theorem}[section]
\newtheorem{lemma}[theorem]{Lemma}
\newtheorem{proposition}[theorem]{Proposition}
\theoremstyle{remark}

\newcommand{\F}{\mathbb F}
\newcommand{\Z}{\mathbb Z}
\newcommand{\ph}{\varphi}
\newcommand{\units}[1]{{(\Z/#1\Z)^\times}}

\title{Many Representations as Sums of Three Prime Cubes}
\author{Yukai Wang and Xu Zhang$^{*}$\\
\footnote{Email addresses:\ 202217794@mail.sdu.edu.cn (Y. Wang),\ xu\_zhang\_sdu@mail.sdu.edu.cn (X. Zhang).
*Corresponding author}
 School of Mathematics and Statistics\\
Shandong University, Weihai 264209, China}
\date{}
\begin{document}
\maketitle

\begin{abstract}
Let $F_k(n)$ be the number of unordered representations
\[
 n=p_1^k+p_2^k+\cdots +p_k^k
\]
by primes, with repetitions allowed.  Erd\H{o}s stated that
$\limsup_{n\to\infty} F_3(n)=\infty$, but his proof appears not to have been published.
A complete unconditional proof is given. The principal input is the classical Hecke equidistribution theorem for the CM Fermat cubic; the rest of the argument uses standard estimates for primes in arithmetic progressions and elementary counting.
The argument used for \(k=3\) does not extend to the case \(k=4\). Nevertheless, by applying the Green--Tao--Ziegler theorem to the linear forms arising from an admissible binary quartic identity, \(\limsup_{n\to\infty}F_4(n)\ge2\) is shown.
\end{abstract}

\section{Introduction}\label{Introduction}

For integers $k\ge2$ and $n\ge1$, let $F_k(n)$ denote the number of unordered
representations
\[
 n=p_1^k+\cdots+p_k^k
\]
by primes, with repetitions allowed. Erd\H{o}s Problem 979 asks whether
\[
 \limsup_{n\to\infty}F_k(n)=\infty
\]
for every $k\ge2$. The case $k=2$ was proved by Erd\H{o}s~\cite{Erdos1937}. Erd\H{o}s and Graham~\cite[p.~47]{ErdosGraham1980} state that the maximum in the cubic case
tends to infinity, but do not include a proof. The maintained Erd\H{o}s
Problems entry~\cite{ErdosProblems979} attributes the cubic case to an
unpublished argument of Erd\H{o}s and lists the general problem as open. Our
first result gives a complete proof of the cubic assertion.

\begin{theorem}[The cubic case]\label{thm:cubic}
There are integers with arbitrarily many representations as sums of three
cubes of primes. Equivalently,
\[
 \limsup_{n\to\infty}F_3(n)=\infty.
\]
\end{theorem}

Our second result is a partial result for $k=4$.

\begin{theorem}[A quartic partial result]\label{thm:quartic}
There are infinitely many positive integers having at least two genuinely
distinct unordered representations as sums of four fourth powers of distinct
primes. In particular,
\[
 \limsup_{n\to\infty}F_4(n)\ge2.
\]
\end{theorem}

For the proof of Theorem~\ref{thm:cubic}, let $R_3(n)$ count ordered triples
of primes whose cubes sum to $n$. Every unordered representation has between
one and $3!$ orderings, according to whether some of its primes coincide.
Consequently,
\[
 F_3(n)\le R_3(n)\le3!F_3(n).
\]
Thus it is enough to prove that $R_3(n)$ is unbounded.

\paragraph{Notation and standard inputs.}

For a positive integer $m$, let
\[
 \units{m}=\{a\bmod m:(a,m)=1\},
\]
and write $\ph(m)=|\units{m}|$.  We also write
\[
 \pi(X;m,a)=\#\{p\le X:p\text{ prime and }p\equiv a\pmod m\}.
\]
The notation $f\ll g$ or $g\gg f$ means that $|f|\le Cg$ for some constant
$C$, and $f\asymp g$ means that both $f\ll g$ and $g\ll f$.  The dependence
of implied constants on fixed parameters will be indicated when relevant.

The cubic proof uses the prime number theorem, the Chinese remainder theorem, and the Brun--Titchmarsh inequality in the form
\begin{equation}\label{eq:BT}
 \pi(X;q,a)\le\frac{2X}{\ph(q)\log(X/q)}
 \qquad ((a,q)=1,\ q<X),
\end{equation}
as stated, for example, in~\cite[Theorem~6.6]{IwaniecKowalski}.  The only
deeper input in the cubic proof is a consequence of Hecke equidistribution
for a CM elliptic curve, stated precisely in Section~\ref{sec:local}.

The quartic proof uses Green and Tao's qualitative theorem for
finite-complexity systems of linear forms in the primes, specifically
\cite[Corollary~1.9]{GreenTaoLEP}. In that paper the result was conditional on
the Inverse Gowers-norm Conjecture $\mathrm{GI}(s)$ and the M\"obius and
Nilsequences Conjecture $\mathrm{MN}(s)$. Green and Tao subsequently proved
$\mathrm{MN}(s)$ for every $s$ in~\cite{GreenTaoMobius}, while Green, Tao, and
Ziegler proved the remaining cases of $\mathrm{GI}(s)$
in~\cite{GreenTaoZiegler}. Thus the form of
\cite[Corollary~1.9]{GreenTaoLEP} used here is unconditional.

\paragraph{Organization of the article.}
The article is organized as follows. Section~\ref{sec:popular} proves an
elementary lemma concerning primes in residue classes.
Section~\ref{sec:local} obtains the required local solution counts from Hecke
equidistribution for the Fermat cubic. Section~\ref{sec:modulus} constructs a
balanced modulus. Section~\ref{sec:global} combines these ingredients to prove
Theorem~\ref{thm:cubic}. Finally, Section~\ref{sec:quartic} proves Theorem~\ref{thm:quartic} and then explains why the argument used in the cubic case does not extend to $k=4$.

\section{Popular residue classes}\label{sec:popular}

This section establishes an elementary lemma concerning primes in reduced
residue classes.  Given a squarefree modulus written as a product of pairwise
coprime factors, the lemma shows that at least one factor has many primes in
almost all of its reduced residue classes.

\begin{lemma}[Popular residue classes]\label{lem:popular}
Fix $0<\delta<1$ and choose a fixed integer $r$ so large that
\[
 4(1-\delta)^r<\frac18.
\]
Let $A=A_1\cdots A_r$ be squarefree, where the $A_j$ are pairwise coprime, and
put $X=A^2$.  For all sufficiently large $A$, some $A_j$ has a set
\[
 S_j\subseteq\units{A_j},\qquad |S_j|\ge(1-\delta)\ph(A_j),
\]
such that every $a\in S_j$ contains at least
\[
 g_j=\left\lfloor\frac{\eta X}{\ph(A_j)\log X}\right\rfloor,
 \qquad \eta=\frac1{8r},
\]
primes $p\le X$, $p\nmid A$, with $p\equiv a\pmod{A_j}$.
\end{lemma}

\begin{proof}
For each $j$, call a reduced residue class $a\bmod A_j$ \emph{popular} if it
contains at least
\[
 \frac{\eta X}{\ph(A_j)\log X}
\]
primes $p\le X$ with $p\nmid A$. The reduced residue classes modulo $A_j$ partition the primes $p\le X$ with $p\nmid A$. Hence their average occupancy is
\[
\frac{1}{\ph(A_j)}\#\{p\le X:p\text{ prime and }p\nmid A\}
\sim \frac{X}{\ph(A_j)\log X},
\]
by the prime number theorem.

Suppose, for a contradiction, that every $A_j$ has fewer than
$(1-\delta)\ph(A_j)$ popular reduced classes.  We divide the primes $p\le X$,
$p\nmid A$, into two sets.

First consider primes which lie in an unpopular class for at least one
$A_j$.  For a fixed $j$, all unpopular classes together contain fewer than
\[
 \ph(A_j)\frac{\eta X}{\ph(A_j)\log X}
 =\frac{\eta X}{\log X}
\]
such primes.  Summing over $j$ and using $\eta=1/(8r)$ shows that their union
contains at most
\begin{equation}\label{eq:unpopular-primes}
 r\frac{\eta X}{\log X}=\frac{X}{8\log X}
\end{equation}
primes. This is a union bound; a prime lying in several unpopular classes is counted more than once in the upper estimate.

Every remaining prime lies in a popular class modulo every $A_j$.  By the
Chinese remainder theorem,
\[
 \units{A}\cong\prod_{j=1}^r\units{A_j}.
\]
Our contradictory assumption therefore leaves fewer than
\[
 \prod_{j=1}^r(1-\delta)\ph(A_j)=(1-\delta)^r\ph(A)
\]
possible reduced classes modulo $A$ for the remaining primes.  Since
$A=X^{1/2}<X$, the Brun--Titchmarsh bound~\eqref{eq:BT} gives, for each such
class,
\[
 \pi(X;A,a)\le\frac{2X}{\ph(A)\log(X/A)}
 =\frac{4X}{\ph(A)\log X}.
\]
Hence the number of remaining primes is less than
\[
 (1-\delta)^r\ph(A)\frac{4X}{\ph(A)\log X}
 =4(1-\delta)^r\frac{X}{\log X}<\frac{X}{8\log X}.
\]
Together with~\eqref{eq:unpopular-primes}, this would leave fewer than
$X/(4\log X)$ primes $p\le X$ with $p\nmid A$.

This contradicts the prime number theorem.  Indeed, since $X=A^2$, every
prime divisor of $A$ is at most $A=\sqrt X$, and the number of such divisors is at
most
\[
 \frac{\log A}{\log 2}=O(\log X).
\]
Consequently,
\[
 \#\{p\le X:p\text{ prime and }p\nmid A\}
 =\pi(X)-O(\log X)
 \sim\frac{X}{\log X}.
\]
For all sufficiently large $X$, this quantity is greater than
$X/(4\log X)$, contradicting the preceding upper bound.  Hence some $A_j$
has at least $(1-\delta)\ph(A_j)$ popular reduced residue classes.

\end{proof}

\section{The local excess on the Fermat cubic}\label{sec:local}

\subsection{The curve and its Frobenius traces}

Let $E/\mathbb Q$ be the projective Fermat cubic
\[
 E:\quad x^3+y^3+z^3=0.
\]
For every prime $\ell\ne3$, the same equation defines its reduction
$E_\ell/\F_\ell$.  Its set of $\F_\ell$-rational points is
\[
E_\ell(\F_\ell)=\{[x:y:z]\in\mathbb P^2(\F_\ell):x^3+y^3+z^3=0\}.
\]
The partial derivatives $3x^2$, $3y^2$, and $3z^2$ cannot vanish simultaneously at a projective point, so $E_\ell$ is smooth. The genus formula for a smooth plane cubic gives genus one, and $[1:-1:0]\in E(\mathbb Q)$ supplies an origin.  Thus $E$ is an elliptic curve
over $\mathbb Q$ with good reduction at every $\ell\ne3$; see
\cite[Chapter~III, \S1]{SilvermanAEC}.  We write
\begin{equation}\label{eq:frob-trace}
 \#E_\ell(\F_\ell)=\ell+1-a_\ell.
\end{equation}
Here $a_\ell$ is the Frobenius trace.  The Hasse bound
$|a_\ell|\le2\sqrt\ell$ indicates the natural scale on which the point count
fluctuates around $\ell+1$.

The curve $E$ has complex multiplication by the ring of integers
$\mathcal O_K=\Z[(1+\sqrt{-3})/2]$ of $K=\mathbb Q(\sqrt{-3})$; see
\cite[Chapter~II]{SilvermanCM}.  We isolate the precise consequence of Hecke
equidistribution needed below.

\begin{proposition}[CM equidistribution input]\label{prop:CM-input}
There are constants $c,c_0>0$ such that
\begin{equation}\label{eq:favorable-count}
 \#\{\ell\le Y:\ \ell\equiv1\pmod3,\ a_\ell\le-c\sqrt\ell\}
 \ge c_0\frac{Y}{\log Y}
\end{equation}
for all sufficiently large $Y$.
\end{proposition}

\begin{proof}
For every prime $\ell\ne3$, the condition $\ell\equiv1\pmod3$ is equivalent
to splitting in $\mathbb Q(\sqrt{-3})$.  The CM Sato--Tate theorem states that,
as $\ell$ ranges over these split primes, the normalized traces
\[
 x_\ell=\frac{a_\ell}{\sqrt\ell}
\]
are equidistributed on $[-2,2]$ with respect to the probability measure
\[
 d\mu_{\mathrm{CM}}(x)=\frac{dx}{\pi\sqrt{4-x^2}};
\]
see~\cite[Section~2.4]{SutherlandSatoTate}.  Thus, for every interval
$I\subset(-2,2)$,
\[
 \frac{\#\{\ell\le Y:\ \ell\equiv1\pmod3,\ x_\ell\in I\}}
 {\#\{\ell\le Y:\ \ell\equiv1\pmod3\}}
 \longrightarrow\mu_{\mathrm{CM}}(I).
\]
The prime number theorem in arithmetic progressions gives
\[
 \#\{\ell\le Y:\ \ell\equiv1\pmod3\}
 \sim\frac12\frac{Y}{\log Y}.
\]
Taking $I=[-3/2,-1]$, which has positive $\mu_{\mathrm{CM}}$-measure, we
therefore obtain
\[
 \#\{\ell\le Y:\ \ell\equiv1\pmod3,\ x_\ell\in I\}
 \sim\frac{\mu_{\mathrm{CM}}(I)}2\frac{Y}{\log Y}.
\]
Every prime counted on the left satisfies $a_\ell\le-\sqrt\ell$.  Hence
\eqref{eq:favorable-count} holds with $c=1$ and any sufficiently small
$c_0>0$.
\end{proof}

We call the primes counted in~\eqref{eq:favorable-count} \emph{favorable}.
Only Proposition~\ref{prop:CM-input} from the CM theory will be used below.

\subsection{Counting affine nonzero solutions}

For a favorable prime $\ell$, in particular $\ell\equiv1\pmod3$, define
\[
 N_\ell=\#\{(x,y,z)\in(\F_\ell^\times)^3:x^3+y^3+z^3=0\}.
\]
We pass from the projective point count~\eqref{eq:frob-trace} to this number of
 affine nonzero solutions.

If $x=0$, then a projective solution has $z\ne0$ and
\[
 (y/z)^3=-1.
\]
Because $\F_\ell^\times$ is cyclic of order $\ell-1$ and
$3\mid(\ell-1)$, this equation has exactly three solutions.  The same count
holds when $y=0$ or $z=0$.  These three sets of projective points are disjoint:
two zero coordinates would force the third one to vanish.  Hence there are
exactly nine projective points with a zero coordinate.  The number of
projective points with all coordinates nonzero is therefore
\[
 \#E_\ell(\F_\ell)-9=\ell-8-a_\ell.
\]
Each such projective point has exactly $\ell-1$ nonzero vector
representatives, one for every common scaling factor in $\F_\ell^\times$.
Consequently
\begin{equation}\label{eq:Nell}
 N_\ell=(\ell-1)(\ell-8-a_\ell).
\end{equation}

\subsection{Normalized local density and its excess}

There are $(\ell-1)^3$ triples in $(\F_\ell^\times)^3$.  If the value of
$x^3+y^3+z^3$ behaved uniformly modulo $\ell$, the equation with right-hand
side zero would be expected to hold for a proportion $1/\ell$ of them.  The
corresponding random-model count is $(\ell-1)^3/\ell$.  This motivates the
normalized local density
\begin{equation}\label{eq:sigma-def}
 \sigma_\ell=\frac{\ell N_\ell}{(\ell-1)^3}
 =\frac{\ell(\ell-8-a_\ell)}{(\ell-1)^2}.
\end{equation}
Thus $\sigma_\ell=1$ is the random baseline, whereas $\sigma_\ell>1$ is a
local excess of solutions.

Subtracting the baseline from~\eqref{eq:sigma-def} gives the exact identity
\begin{equation}\label{eq:sigma-minus-one}
 \sigma_\ell-1
 =\frac{\ell(\ell-8-a_\ell)-(\ell-1)^2}{(\ell-1)^2}
 =\frac{-6\ell-a_\ell\ell-1}{(\ell-1)^2}.
\end{equation}
For a favorable prime, $-a_\ell\ell\ge c\ell^{3/2}$.  Thus there are
$Y_0$ and $c_1>0$ such that every favorable prime $\ell\ge Y_0$ satisfies
\begin{equation}\label{eq:sigma-gain}
 \sigma_\ell\ge1+\frac{c_1}{\sqrt\ell}.
\end{equation}

Fix $Y_0$ sufficiently large for~\eqref{eq:sigma-gain}.  Since $Y_0$ is fixed,
\eqref{eq:favorable-count} implies
\[
 \#\{\ell:\ Y_0\le\ell\le Y,\ \ell\text{ favorable}\}\gg\frac{Y}{\log Y}.
\]
As $\ell^{-1/2}\ge Y^{-1/2}$ for every $\ell\le Y$, it follows directly that
\begin{equation}\label{eq:weight-sum}
 \sum_{\substack{Y_0\le\ell\le Y\\\ell\ \mathrm{favorable}}}\ell^{-1/2}\ge\frac1{\sqrt Y}\#\{\ell:\ Y_0\le\ell\le Y,\ \ell\text{ favorable}\}\gg\frac{\sqrt Y}{\log Y}.
\end{equation}
For small $u>0$, $\log(1+u)\ge u/2$.  Enlarging $Y_0$ if necessary and using
\eqref{eq:sigma-gain}, we obtain
\begin{align}
 \log\prod_{\substack{Y_0\le\ell\le Y\\\ell\ \mathrm{favorable}}}\sigma_\ell
 &=\sum_{\substack{Y_0\le\ell\le Y\\\ell\ \mathrm{favorable}}}\log\sigma_\ell\notag\\
 &\ge\frac{c_1}{2}\sum_{\substack{Y_0\le\ell\le Y\\\ell\ \mathrm{favorable}}}\ell^{-1/2}
 \gg\frac{\sqrt Y}{\log Y}.
\end{align}
Hence
\begin{equation}\label{eq:product-gain}
 \prod_{\substack{Y_0\le\ell\le Y\\\ell\ \mathrm{favorable}}}\sigma_\ell
 \ge\exp\!\left(c_2\frac{\sqrt Y}{\log Y}\right)
\end{equation}
for some $c_2>0$.

\section{A balanced modulus}\label{sec:modulus}

From now on fix $\delta=1/12$ and choose the fixed integer $r$ as in
Lemma~\ref{lem:popular}.

Let $\mathcal L(Y)$ be the set of favorable primes in $[Y_0,Y]$ and put
\[
 A=\prod_{\ell\in\mathcal L(Y)}\ell.
\]
We must split this product into the fixed number $r$ of pairwise coprime
factors required by Lemma~\ref{lem:popular}.

The lemma selects one of the factors $A_1,\ldots,A_r$, but does not specify
which one. We must therefore construct the factors so that each $A_j$ has a
large local-density gain. The correct quantity to balance is not the number of prime factors but the weight $\ell^{-1/2}$ that controls $\log\sigma_\ell$.

\begin{lemma}[Greedy balancing]\label{lem:balancing}
Let $w_1,\ldots,w_m>0$, let $W=\sum_iw_i$, and let
$w_{\max}=\max_iw_i$.  The weights can be partitioned into $r$ groups with
group sums $W_1,\ldots,W_r$ satisfying
\[
 W_j\ge\frac Wr-\frac{r-1}{r}w_{\max}\qquad(1\le j\le r).
\]
\end{lemma}

\begin{proof}
Process the weights one at a time, always assigning the next weight to a group
whose current sum is minimal.  Let $M$ be a group with largest final sum, and
let $w$ be the last weight assigned to $M$.  Write $s$ for the sum of $M$
immediately before $w$ was assigned.  At that moment $M$ had minimal current
sum, so every other group had current sum at least $s$.  Since group sums can
only increase, every other group has final sum at least $s$, whereas the final
sum of $M$ is $s+w\le s+w_{\max}$.  Thus the largest and smallest final group
sums differ by at most $w_{\max}$.

Let $W_{\min}$ be the smallest final group sum.  By the preceding argument, one group has sum $W_{\min}$, and each of the other $r-1$ groups has sum at most
$W_{\min}+w_{\max}$.  Therefore
\[
 W\le W_{\min}+(r-1)(W_{\min}+w_{\max})
 =rW_{\min}+(r-1)w_{\max}.
\]
It follows that
\[
 W_j\ge W_{\min}\ge\frac Wr-\frac{r-1}{r}w_{\max}
\]
for every $j$.
\end{proof}

Apply Lemma~\ref{lem:balancing} with $w_\ell=\ell^{-1/2}$ for
$\ell\in\mathcal L(Y)$.  Let $\mathcal L_1,\ldots,\mathcal L_r$ be the
resulting groups and set
\[
 A_j=\prod_{\ell\in\mathcal L_j}\ell,
 \qquad A=A_1\cdots A_r.
\]
The factors are squarefree and pairwise coprime.  By~\eqref{eq:weight-sum}, the
total weight $W$ tends to infinity and is $\gg\sqrt Y/\log Y$, whereas
$w_{\max}\le Y_0^{-1/2}$ is fixed.  Therefore, for all sufficiently large
$Y$, every group satisfies
\begin{equation}\label{eq:group-weight}
 \sum_{\ell\mid A_j}\ell^{-1/2}\ge\frac{W}{2r}
 \gg\frac{\sqrt Y}{\log Y}.
\end{equation}
Using $\log\sigma_\ell\ge(c_1/2)\ell^{-1/2}$ and
\eqref{eq:group-weight}, we conclude that every $j$ satisfies
\begin{equation}\label{eq:group-gain}
 \prod_{\ell\mid A_j}\sigma_\ell
 \ge\exp\!\left(c_3\frac{\sqrt Y}{\log Y}\right)
\end{equation}
with a constant $c_3>0$ independent of $j$ and $Y$.

Put $X=A^2$.  Lemma~\ref{lem:popular} selects an index $j$ and a set
\[
 S_j\subseteq\units{A_j},\qquad |S_j|\ge(1-\delta)\ph(A_j),
\]
such that every class in $S_j$ contains at least
\[
 g_j=\left\lfloor\frac{\eta X}{\ph(A_j)\log X}\right\rfloor
\]
primes up to $X$ not dividing $A$.  The balancing step is what makes this
selection compatible with the local construction: although $j$ is unknown in
advance, the gain~\eqref{eq:group-gain} holds for every candidate.

\section{From local solutions to prime representations}\label{sec:global}

\subsection{Chinese remaindering of the local densities}

Define
\[
 N(A_j)=\#\{(x,y,z)\in\units{A_j}^3:
 x^3+y^3+z^3\equiv0\pmod{A_j}\}.
\]
Because $A_j$ is a product of distinct favorable primes, the Chinese
remainder theorem identifies a solution modulo $A_j$ with an independent
choice of a solution modulo every $\ell\mid A_j$.  Thus
\[
 N(A_j)=\prod_{\ell\mid A_j}N_\ell.
\]
The same factorization gives
\[
 A_j=\prod_{\ell\mid A_j}\ell,
 \qquad
 \ph(A_j)=\prod_{\ell\mid A_j}(\ell-1).
\]
Combining these identities with~\eqref{eq:sigma-def} yields the exact density
factorization
\begin{equation}\label{eq:CRT-density}
 \frac{A_jN(A_j)}{\ph(A_j)^3}
 =\prod_{\ell\mid A_j}\frac{\ell N_\ell}{(\ell-1)^3}
 =\prod_{\ell\mid A_j}\sigma_\ell.
\end{equation}

\subsection{Restricting all three coordinates to popular classes}

We next show that restricting all three coordinates to $S_j$ discards at most
a proportion $3\delta$ of the congruence solutions. Fix a unit $x\bmod A_j$.  Simultaneous multiplication by $x^{-1}$ gives a bijection
\[
 (x,y,z)\longmapsto(1,yx^{-1},zx^{-1})
\]
from the solutions with first coordinate $x$ to those with first coordinate
$1$.  Hence every first-coordinate fiber has the same size.  There are
$\ph(A_j)$ possible first coordinates and $N(A_j)$ solutions in total, so each
fiber has exactly
\[
 \frac{N(A_j)}{\ph(A_j)}
\]
elements.  By symmetry, the same statement holds for the second and third
coordinates.

At most $\delta\ph(A_j)$ first coordinates lie outside $S_j$.  Since each
first-coordinate fiber contains exactly $N(A_j)/\ph(A_j)$ solutions, the
number of solutions with $x\notin S_j$ is at most
\[
 \delta\ph(A_j)\frac{N(A_j)}{\ph(A_j)}=\delta N(A_j).
\]
The same bound holds for $y\notin S_j$ and for $z\notin S_j$.  A union bound
therefore gives
\begin{equation}\label{eq:popular-solutions}
 \#\{(x,y,z)\in S_j^3:x^3+y^3+z^3\equiv0\pmod{A_j}\}
 \ge(1-3\delta)N(A_j).
\end{equation}

\subsection{Choosing primes and applying the pigeonhole principle}

For every congruence solution counted in~\eqref{eq:popular-solutions}, each of
its three classes contains at least $g_j$ primes up to $X$.  Choosing one prime
from each class produces at least $g_j^3$ ordered prime triples.  We therefore obtain at least
\begin{equation}\label{eq:prime-triples}
 (1-3\delta)N(A_j)g_j^3
\end{equation}
ordered prime triples $(p,q,s)$ with $p,q,s\le X$ satisfying
\[
 p^3+q^3+s^3\equiv0\pmod{A_j}.
\]

Every such sum is positive and at most $3X^3$.  Since it is divisible by
$A_j$, it can take at most
\[
 \left\lfloor\frac{3X^3}{A_j}\right\rfloor\le\frac{3X^3}{A_j}
\]
values.  The pigeonhole principle applied to~\eqref{eq:prime-triples} gives an
integer $n$ for which
\begin{equation}\label{eq:pigeonhole}
 R_3(n)\ge\frac{(1-3\delta)N(A_j)g_j^3A_j}{3X^3}.
\end{equation}

It remains to simplify this bound.  Since $A_j\le A=X^{1/2}$ and $\ph(A_j)\le A_j$,
\[
 \frac{X}{\ph(A_j)\log X}\ge\frac{X^{1/2}}{\log X}\longrightarrow\infty.
\]
The convergence is uniform in $j$. Hence, for all sufficiently large $Y$,
\[
 g_j=\left\lfloor\frac{\eta X}{\ph(A_j)\log X}\right\rfloor
 \ge\frac{\eta X}{2\ph(A_j)\log X}
 \gg\frac{X}{\ph(A_j)\log X}.
\]
Substitution in~\eqref{eq:pigeonhole}, followed by~\eqref{eq:CRT-density}, gives
\begin{align}
 R_3(n)
 &\gg N(A_j)\frac{X^3}{\ph(A_j)^3(\log X)^3}\frac{A_j}{X^3}\notag\\
 &=\frac{A_jN(A_j)}{\ph(A_j)^3(\log X)^3}
 =\frac1{(\log X)^3}\prod_{\ell\mid A_j}\sigma_\ell.
 \label{eq:final-lower-bound}
\end{align}

Finally,
\[
 \log A=\sum_{\ell\in\mathcal L(Y)}\log\ell
 \le\sum_{\ell\le Y}\log\ell=O(Y),
\]
where the last estimate follows, for instance, from the prime number theorem.
Since $X=A^2$, we have $\log X=O(Y)$.  Combining
\eqref{eq:group-gain} and~\eqref{eq:final-lower-bound} yields
\[
 R_3(n)\gg Y^{-3}\exp\!\left(c_3\frac{\sqrt Y}{\log Y}\right),
\]
which tends to infinity as $Y\to\infty$. Thus the construction produces
integers $n$ for which $R_3(n)$ is arbitrarily large. These integers must
themselves be unbounded, since any fixed integer has only finitely many prime
representations. Finally, $F_3(n)\ge R_3(n)/3!$, proving
Theorem~\ref{thm:cubic}.

\section{The quartic case}\label{sec:quartic}

We first prove Theorem~\ref{thm:quartic} by combining an admissible binary quartic identity with the theorem of Green and Tao on linear forms
in the primes~\cite[Corollary~1.9]{GreenTaoLEP}. Then in Subsection~\ref{subsec:quartic-obstruction},
we explain why the local-density argument used in the cubic case fails in the quartic case.

\subsection{An admissible quartic identity}

For integers $x,y$, define
\begin{align*}
 L_1&=x-6y, & L_2&=4x-y, & L_3&=5x+4y, & L_4&=6x-5y,\\
 M_1&=x-4y, & M_2&=4x+5y, & M_3&=5x-6y, & M_4&=6x-y.
\end{align*}

\begin{lemma}[Quartic identity]\label{lem:identity}
One has the polynomial identity
\begin{equation}\label{eq:quartic-identity}
 L_1^4+L_2^4+L_3^4+L_4^4=M_1^4+M_2^4+M_3^4+M_4^4.
\end{equation}
Both sides are equal to
\begin{equation}\label{eq:common-quartic}
 Q(x,y)=2178x^4-2600x^3y+8112x^2y^2-2600xy^3+2178y^4.
\end{equation}
\end{lemma}
\begin{proof}
Expanding the two sums and collecting coefficients shows that each equals the
polynomial $Q(x,y)$ in \eqref{eq:common-quartic}, which proves
\eqref{eq:quartic-identity}.
\end{proof}

\begin{lemma}[Admissibility]\label{lem:admissible}
The system
\[
 \Psi=(L_1,L_2,L_3,L_4,M_1,M_2,M_3,M_4)
\]
is admissible: for every prime $q$ there is $(x,y)\in\F_q^2$ for which none of
the eight values is zero.  Moreover, the eight forms are pairwise
nonproportional over $\mathbb Q$.
\end{lemma}

\begin{proof}
For $q=2,3,5,7$, suitable pairs $(x,y)$ are displayed below.  The last column
lists the eight nonzero residues in the order used in the statement.
\[
\begin{array}{ccl}
\toprule
q&(x,y)&\Psi(x,y)\pmod q\\
\midrule
2&(1,1)&(1,1,1,1,1,1,1,1)\\
3&(1,2)&(1,2,1,2,2,2,2,1)\\
5&(1,2)&(4,2,3,1,3,4,3,4)\\
7&(0,1)&(1,6,4,2,3,5,1,6)\\
\bottomrule
\end{array}
\]
Now let $q\ge11$. The zero set of any one of the eight linear forms is
a line through the origin in $\F_q^2$, and hence contains exactly $q$ points.
Therefore the union of their zero sets contains at most $8q$ points. Since
\[8q<q^2=|\F_q^2|,\]
there is a pair $(x,y)$ outside this union. At that pair, none of the eight forms vanishes. Together with the four explicit checks, this proves admissibility for every prime $q$.

It remains to verify pairwise nonproportionality over $\mathbb Q$. Two forms $ax+by$ and $a'x+b'y$ are proportional if and only if $b/a=b'/a'$. For the eight forms, these ratios are
\[
 -6,\quad-\frac14,\quad\frac45,\quad-\frac56,\quad
 -4,\quad\frac54,\quad-\frac65,\quad-\frac16
\]
respectively. They are all distinct, so no two of the eight forms are proportional.

\end{proof}

\subsection{Prime values of all eight forms}

To prove Theorem~\ref{thm:quartic}, we use the following qualitative
consequence of Green and Tao's theorem~\cite[Corollary~1.9]{GreenTaoLEP}. As explained in
Section~\ref{Introduction}, the theorem was originally conditional on two
conjectures. Subsequent work of Green and Tao~\cite{GreenTaoMobius} and of
Green, Tao, and Ziegler~\cite{GreenTaoZiegler} established those conjectures,
so the form used here is unconditional.

\begin{proposition}[A consequence of Green--Tao--Ziegler] \label{prop:GTZ}
Let $\psi_1,\ldots,\psi_t$ be nonzero integer linear forms in $d$ variables,
no two of which are proportional over $\mathbb Q$. Suppose that for every
prime $q$ there is $\mathbf n\in\F_q^d$ such that
\[
 \psi_1(\mathbf n)\cdots\psi_t(\mathbf n)\ne0.
\]
If $\mathcal C\subset\mathbb R^d$ is a nonempty open convex cone (that is, an open
set invariant under multiplication by positive scalars) on which all the
forms are positive, then there are infinitely many
$\mathbf n\in\mathcal C\cap\mathbb Z^d$ for which
\[
 \psi_1(\mathbf n),\ldots,\psi_t(\mathbf n)
\]
are all prime.
\end{proposition}

\begin{proof}[Proof of Theorem~\ref{thm:quartic}]
Let
\[
 \mathcal C=\{(x,y)\in\mathbb R^2:y>0,\ 10y<x<11y\}.
\]
This is a nonempty open convex cone. For every $(x,y)\in\mathcal C$, it is straightforward to see
\[
 0<L_1<L_2<L_3<L_4
 \qquad\text{and}\qquad
 0<M_1<M_3<M_2<M_4.
\]
Thus all eight forms are positive on $\mathcal C$, and the four forms in each quadruple take pairwise distinct values there.

Lemma~\ref{lem:admissible} verifies the admissibility and pairwise
nonproportionality hypotheses of
Proposition~\ref{prop:GTZ}. Applying Proposition~\ref{prop:GTZ} to the eight forms therefore gives infinitely many integer pairs
$(x,y)\in\mathcal C$ for which
\[
 L_1(x,y),\ldots,L_4(x,y),M_1(x,y),\ldots,M_4(x,y)
\]
are all prime. For every such pair, Lemma~\ref{lem:identity} gives
\[
 L_1(x,y)^4+\cdots+L_4(x,y)^4
       =M_1(x,y)^4+\cdots+M_4(x,y)^4.
\]
Thus both sides are representations by fourth powers of four distinct
primes. Moreover, the inequalities above show that $L_4(x,y)$ and $M_4(x,y)$
are the largest primes in the respective quadruples, and
\[
 M_4(x,y)-L_4(x,y)=4y>0,
\]
which implies that the two unordered representations are distinct. This proves the
theorem.

\end{proof}

\subsection{Why the cubic argument does not extend}\label{subsec:quartic-obstruction}

The cubic proof ends with the lower bound
\[
 \frac{1}{(\log X)^3}\prod_{\ell\mid A_j}\sigma_\ell
\]
from~\eqref{eq:final-lower-bound}.  Since favorable
primes satisfy $\sigma_\ell\ge 1+c\ell^{-1/2}$, the product of the local
densities grows rapidly enough to overcome the factor $(\log X)^3$; see
\eqref{eq:group-gain}.

We now turn to the case $k=4$.  As in the cubic case, for an odd prime $\ell$, let
\[
 N_{4}(\ell)
 =\#\{(x_1,x_2,x_3,x_4)\in(\F_\ell^\times)^4:x_1^4+x_2^4+x_3^4+x_4^4=0\},
\]
and set
\[
 \sigma_{4}(\ell)=\frac{\ell N_{4}(\ell)}{(\ell-1)^4}.
\]
In contrast with the cubic estimate~\eqref{eq:sigma-gain}, the quartic density
satisfies the following bound.

\begin{lemma}[Quartic local density]\label{lem:quartic-density}
Uniformly over odd primes $\ell$,
\begin{equation}\label{eq:local-ceiling}
 \sigma_4(\ell)=1+O\!\left(\frac1\ell\right).
\end{equation}
\end{lemma}

\begin{proof}
Let $S_\ell\subset\mathbb P^3_{\F_\ell}$ be the projective surface defined by
\[
 x_1^4+x_2^4+x_3^4+x_4^4=0.
\]
Its partial derivatives are $4x_1^3,\ldots,4x_4^3$.  Since $\ell$ is odd,
they cannot vanish simultaneously at a projective point, so $S_\ell$ is
smooth.

A smooth quartic surface in $\mathbb P^3$ is a K3 surface.  In particular,
its odd cohomology vanishes and its second Betti number is $22$.  The
Grothendieck--Lefschetz trace formula therefore gives
\[
 \#S_\ell(\F_\ell)=1+\ell^2+T_\ell,
\]
where $T_\ell$ is the Frobenius trace on the $22$-dimensional middle
cohomology.  Deligne's bounds~\cite{DeligneWeil} give $|T_\ell|\le22\ell$.
Consequently,
\[
 \#S_\ell(\F_\ell)=\ell^2+O(\ell),
\]
with an absolute implied constant.

This estimate counts all projective points, including those with a zero
coordinate.  Let
\[
 B_\ell=\{[x_1:x_2:x_3:x_4]\in S_\ell(\F_\ell):x_1x_2x_3x_4=0\}.
\]
For each $i$, the coordinate section $S_\ell\cap\{x_i=0\}$ is a smooth plane
quartic curve and hence has genus $3$.  The Weil bound for curves gives
\[
 \#(S_\ell\cap\{x_i=0\})(\F_\ell)=\ell+1+O(\sqrt\ell)=O(\ell).
\]
Taking the union of the four coordinate sections, we obtain
\[
 \#B_\ell=O(\ell).
\]
It follows that the number of points of $S_\ell(\F_\ell)$ with all four
coordinates nonzero is
\[
 \#S_\ell(\F_\ell)-\#B_\ell=\ell^2+O(\ell).
\]

Each such projective point has exactly $\ell-1$ vector representatives in
$(\F_\ell^\times)^4$, and these representatives are precisely the vectors
counted by $N_4(\ell)$.  Consequently,
\[
 N_4(\ell)=(\ell-1)(\ell^2+O(\ell))=\ell^3+O(\ell^2).
\]
Finally,
\[
 \sigma_4(\ell)=\frac{N_4(\ell)}{(\ell-1)^4/\ell}
 =1+O\!\left(\frac1\ell\right),
\]
as required.
\end{proof}

To see what the cubic argument gives in four variables, let $A_j$, $S_j$,
and $g_j$ be supplied by Lemma~\ref{lem:popular}, and set
\[
 N_4(A_j)=\#\{(x_1,x_2,x_3,x_4)\in\units{A_j}^4:x_1^4+x_2^4+x_3^4+x_4^4\equiv0\pmod{A_j}\}.
\]
Since $A_j$ is squarefree, the Chinese remainder theorem gives
\[
 N_4(A_j)=\prod_{\ell\mid A_j}N_4(\ell).
\]
Using also
\[
 A_j=\prod_{\ell\mid A_j}\ell,
 \qquad
 \ph(A_j)=\prod_{\ell\mid A_j}(\ell-1),
\]
we obtain
\begin{equation}\label{eq:global-local-ceiling}
 \frac{A_j N_4(A_j)}{\ph(A_j)^4}
 =\prod_{\ell\mid A_j}\frac{\ell N_4(\ell)}{(\ell-1)^4}
 =\prod_{\ell\mid A_j}\sigma_4(\ell).
\end{equation}

The fiber argument used in the cubic case applies to each of the four
coordinates.  Restricting all four coordinates to $S_j$ therefore leaves at
least
\[
 (1-4\delta)N_4(A_j)
\]
solutions modulo $A_j$.  Since $\delta=1/12$, this is a positive proportion.
Repeating the prime-selection and pigeonhole argument leading
to~\eqref{eq:pigeonhole}, we obtain some $n$ for which
\begin{equation}\label{eq:quartic-pigeonhole}
 R_4(n)\ge\frac{(1-4\delta)N_4(A_j)g_j^4A_j}{4X^4}.
\end{equation}
Here $R_4(n)$ counts ordered representations of $n$ as a sum of four fourth
powers of primes, so $F_4(n)\ge R_4(n)/4!$.  As before, uniformly in $j$,
\[
 g_j\gg\frac{X}{\ph(A_j)\log X}.
\]
Together with~\eqref{eq:global-local-ceiling}, this gives
\begin{equation}\label{eq:quartic-analogue-bound}
 R_4(n)
 \gg\frac{A_jN_4(A_j)}{\ph(A_j)^4(\log X)^4}
 =\frac{1}{(\log X)^4}\prod_{\ell\mid A_j}\sigma_4(\ell).
\end{equation}

On the other hand, Lemma~\ref{lem:quartic-density} implies that, for some absolute constant $C>0$,
\[
 \prod_{\ell\mid A_j}\sigma_4(\ell)
 \ll\prod_{\ell\mid A_j}\left(1+\frac C\ell\right)
 \ll(\log\log A_j)^C,
\]
where the last estimate is a standard consequence of Mertens' theorem. Thus
\begin{equation}\label{eq:quartic-method-bound}
 \frac{1}{(\log X)^4}\prod_{\ell\mid A_j}\sigma_4(\ell)
 \ll\frac{(\log\log A_j)^C}{(\log X)^4}
 \le\frac{(\log\log A)^C}{(2\log A)^4}.
\end{equation}
Here $X=A^2$ and $A_j\le A$.  The last expression tends to zero as
$A\to\infty$, so the quartic analogue of the cubic argument cannot prove that \(F_4(n)\) is unbounded.

\end{document}